\documentclass{article}

\usepackage{amsmath}
\usepackage{amssymb}
\usepackage{amsthm}
\usepackage{graphicx}
\usepackage{amscd}
\usepackage{epic, eepic}
\usepackage{url}
\usepackage{color}
\usepackage[utf8]{inputenc} 
\usepackage{comment}
\usepackage{enumerate}   
\usepackage{epstopdf}
\usepackage[colorinlistoftodos,prependcaption,textsize=small]{todonotes}

\theoremstyle{plain}

\newtheorem{theorem}{\bf Theorem}[section]
\newtheorem{lemma}[theorem]{\bf Lemma}
\newtheorem{cor}[theorem]{\bf Corollary}

\newtheorem{problem}[theorem]{\bf Problem}
\newtheorem{prop}[theorem]{\bf Proposition}
\newtheorem{conj}[theorem]{\bf Conjecture}

\newtheorem{nota}[theorem]{\bf Notation}

\newtheorem{remark}[theorem]{\bf Remark}
\newtheorem{defi}[theorem]{\bf Definition}

\newtheorem{obs}[theorem]{\bf Observation}

\def\exo{\hbox{\rm ex}_\mathcal{OP}}
\def\expo{\hbox{\rm ex}_\mathcal{P}}

\newcommand\cC{{\mathcal C}}

\newcommand\cN{{\mathcal N}}

\theoremstyle{plain}
\newtheorem{Th}{Theorem}[section]
\newtheorem{Lemma}[Th]{Lemma}
\newtheorem{Cor}[Th]{Corollary}

\theoremstyle{definition}

\newtheorem{?}[Th]{Problem}

\title{
\ Generalized Outerplanar Turán numbers and maximum number of  $k$-vertex subtrees}
\author{Dávid Matolcsi\thanks{ E\"otv\"os Lor\'and University, Budapest, Hungary. The author is supported by the  ÚNKP-20-1 New National Excellence Program of the
Ministry of Human Capacities. E-mail: {\tt matolcsidavid@gmail.com}}, 
Zoltán Lóránt Nagy\thanks{ELKH--ELTE Geometric and Algebraic Combinatorics Research Group,
  E\"otv\"os Lor\'and University, Budapest, Hungary. The author is supported by the Hungarian Research Grant (NKFI) No. K 120154, 124950, 134953 and by the János Bolyai Scholarship of the Hungarian Academy of Sciences. 	E-mail: {\tt nagyzoli@caesar.elte.hu}} }
\date{}

\begin{document}

\maketitle

\begin{abstract}

We prove an asymptotic result on the maximum number of $k$-vertex subtrees in binary trees of given order. This problem turns out to be equivalent to determine the maximum number of $k+2$-cycles in  $n$-vertex outerplanar graphs, thus we settle the generalized outerplanar Turán number for all cycles.\\
We also determine the exponential growth of the generalized outerplanar Turán number of paths $P_k$ as a function of $k$ which implies the order of magnitude of the generalized  outerplanar Turán number of arbitrary trees.
The bounds are strongly related to the sequence of Catalan numbers.
\end{abstract}

\section{Introduction}

In a generalized Tur\'an problem, we are given graphs $H$ and $F$ and seek to maximize the
number of copies of $H$ in  $F$-free graphs of order $n$. The restriction on the forbidden subgraph may also concern a whole family of graphs $F\in \mathcal{F}$.
We consider generalized Tur\'an problems where the host graph is outerplanar. Generalized Turán problems were investigated systematically by Alon and Shikhelman \cite{AlonS}, while by Gy\H{o}ri, Paulos, Salia, Tompkins and Zamora  studied the version where the host graph is planar \cite{Gyori}.  They introduced the notation $\expo(n, H)$ for the maximum number of subgraphs $H$ in simple planar graphs on $n$ vertices. Analogously, we introduce the following notation.

\begin{defi}\label{maindef}
$\exo(n, H)$ denotes the maximum number of subgraphs $H$ in simple outerplanar graphs on $n$ vertices.
\end{defi}

 Chartrand and Harary \cite{Char} proved an analogue of Kuratowski's theorem for outerplanar graphs. They characterised outerplanar graphs as those graphs which do not contain a subdivision of $K_4$ or $K_{2,3}$ as a subgraph. Thus  Definition \ref{maindef} is meaningful only for graphs $F$ which does not contain a subdivision of $K_4$ or a $K_{2,3}$, moreover this characterisation implies that $\exo(n, H)$ also refers to a generalized Turán problem.


Problems of this favour, namely to consider a some general family as the host graph and maximise the number of subgraphs of given type has a long history, and it   attained significant interest recently.

Hakimi and Schmeichel \cite{Hakimi} determined the order of magnitude of the planar Turán numbers for  cycles.
Recently, Győri et. al. obtained exact results on planar Turán numbers for short cycles \cite{c4, c5} via the combination of several refined methods. Later, Cox and Martin \cite{Martin} proved asymptotic results for further instances of short cycles and paths. However,  general asymptotic results for long cycles  or paths seems unknown. In fact, the following problem of Győri et. al. \cite{Gyori} (based  on Eppstein's questions \cite{Eppstein}) was also unsolved concerning the 
exponent of the planar or outerplanar Turán number in general, until  recently. 

\begin{problem}
Is it  true that for all $H$, the order of magnitude of the maximum number of copies of $H$
possible in a minor-closed family of $n$-vertex graphs is an integer power of $n$?
\end{problem}

This has been answered affirmatively  very recently for several minor closed classes and graphs $H$ due to  Huynh, Joret and Wood  \cite{Huynh, Huynh2}. Note that
Eppstein \cite{Eppstein} proved  $\exo(n, H)=O(n)$ if and only if $H$ is $2$-connected, moreover for any $2$-connected $H$ there is a linear-time algorithm which enumerates the occurrences of $F$ in any outerplanar graph. (A related result for planar graphs was proved independently by Wormald \cite{Wormald}.)  

Studying outerplanar graphs may be of independent interest as well, since these graphs play an important role in several central problems, like the art gallery problems \cite{Fisk}. Besides, it can serve as a base of comparison for the planar case, or show relations to Hamiltonicity as a Hamiltonian planar triangulation can be obtained from gluing together two maximal outerplanar graphs along the edges of the outer face \cite{Nagy}. Furthermore it may lead to an extension of the Turán-problems in various ways. Either one considers host graphs of bounded Colin de Verdière number $\mu$ (note that $\mu(G)\leq 3$ characterises planar and   $\mu(G)\leq 2$ characterises outerplanar graphs \cite{Colin}), or further minor-closed families.

The first aim of the present paper is to give sharp or asymptotically sharp bounds for certain families of graphs $H$ and describe the extremal graphs as well.
In this direction, we prove our first main result concerning the cycle case $H=C_k$. In contrast with the  planar setting, we can determine the asymptotic result for the maximum number cycle $C_k$ for every value of $k$. 

\begin{theorem}\label{asym}
$$\exo(n, C_k)=(c(k)+o(1))n,$$ where the constant $c(k)$ is determined by the sum $$c(k):=\sum_{r=1}^{k} \frac{t(k,r)}{r}.$$ Here the function $t(k,r)$
is gained from the recursion  $$t(k,r)=\sum_{s=1}^k t(k-r,s)\binom{2s-1}{r-1}$$ with 
 $t(0,r)=0$ $\forall r$,  $t(1,1)=1$, $t(1,r)=0$ if $r\ne 1$, and  $t(k,r)=0$ if $r>k$.
 \end{theorem}

In fact, we prove a more general 
result by describing a family of extremal graphs via applying the results of Andriantiana, Wagner
and Wang \cite{k-tree}, since the problem turns out to be equivalent to
maximize the number of  $k$-vertex subtrees in bounded degree trees.
However we also point out that there is no unique extremal structure in general.

The link between the Turán-type problem and the subgraph enumeration leads to the following equivalent statement, which completes the description of Andriantiana, Wagner
and Wang.

\begin{theorem}
The maximum number of subtrees of size $k-2>0$ in binary trees on $n$ vertices is $(c(k)+o(1))n$, where $c(k)$ is defined as in Theorem \ref{asym}.
\end{theorem}

\begin{table}[h!]
\centering
\begin{tabular}{l|r|r| r|r|r|r|r|r|r|r}
\quad \quad  $k$ & $3$ & $4$ & $5$ & $6$ & $7$ & $8$ & $9$ & $10$ & $11$ & $12$\\\hline
\shortstack{value of $c(k)$ } & $1$ & $1$ & $1.5$ & $2.5$ & $5$ & $10.5$ & $23.75$ & $56.75$ & $141$ & $361.75$ \\ \hline
\end{tabular}
\caption{ Exact asymptotics for $\exo(n, C_k)=c_kn+o(n)$ when $k$ is small}\label{tab:cons}
\end{table}

Our next result determines the growth of the constant $c(k)$. As Table \ref{tab:cons} suggests, the function is exponentially growing.

\begin{theorem}\label{limit}
 $$\lim_{k\rightarrow \infty}\sqrt[k]{c(k)}=4.$$
\end{theorem}

Next we discuss the results of 
 generalized outerplanar number of trees.

\begin{defi}
Let $\ell(T)$ denote the number of leaves of a tree $T$, and in general let $\ell(H)$ denote the number of vertices $v\in V(H)$ of degree one.
\end{defi}


\begin{theorem}[Huynh, Joret, Wood, \cite{Huynh}]\label{outertree} Let $T$ be a tree on $k>2$ vertices. Then $\exo(n, T)= \Theta(n^{\ell(T)})$. 
\end{theorem}
Note that  $\expo(n, P_k)= \Theta(n^{\alpha(P_k)})$ according to \cite{Gyori}, thus in contrast with planar graphs, the order of magnitude does not depend on the length of the path.

Our main result concerning  the case of trees is to  determine the growth of  $\exo(n, P_k)$ in terms of $k$.  This result provides an alternative proof for Theorem \ref{outertree}.

\begin{theorem}\label{paths}
 $ h(k)\binom{n}{2}< \exo(n, P_{k+1})\leq 4^k\binom{n}{2}$, where $\lim_{k\rightarrow \infty}\sqrt[k]{h(k)}=4.$
\end{theorem}


The paper is organised as follows. In the next section we describe the equivalence between determining  $\exo(n, C_k)$ and  a $k$-vertex subtree enumeration problem. Then we  verify the recursion of Theorem \ref{asym} and give an estimate on the multiplicative constant factor obtained from the recursion.
 We prove Theorem \ref{paths} on the number of paths in Section $3$ and derive Theorem \ref{outertree}.
 We also prove some partial results on the extremal graphs, and compare them to  the very recent complete solution of the Cvetkovi\'c-Rowlinson conjecture \cite{BoNing}, which claims that the join graph $K_1\vee P_{n-1}$ attains the maximum spectral radius among all $n$-vertex outerplanar graphs, except for $n=6$.
 


\section{Maximal number of cycles in outerplanar graphs and $k$-vertex subtrees of bounded degree trees}

We start this subsection by deducing the correspondence between $\exo(n, C_k)$ and the problem of maximising the number of subtrees of size $k-2$ in trees having degrees at most $3$.

\begin{defi}[Weak dual] 
The weak dual of an embedded outerplanar graph is defined as follows. Take a vertex for every bounded face of the embedding and connect every pair of vertices corresponding to adjacent faces of the outerplanar graph. 
\end{defi}

Note that the weak dual graph does not contain cycles as a cycle would imply the existence of a vertex of the outerplanar graph which is not lying on the outer face. Also, suppose that the outerplanar graph $G$ is a maximal or triangulated on $n$ vertices. Then $G$ has $2n-3$ edges and the weak dual is a tree on $n-2$ edges, having degrees at most $3$. It is easy to see that every tree having these properties can be obtained by the weak dual of a corresponding maximal outerplanar graph. Trees in which each node has at most $2$ children are called {\em binary trees}.

\begin{defi}[BFS $d$-tree or greedy tree]
Consider a breadth-first search on a $d$-regular infinite tree from an arbitrary root. We call BFS $d$-tree or greedy tree the subtree of the $d$-regular infinite tree induced by the first $n$ found vertices during the search.
\end{defi}

\begin{defi}  Let $g_k(F)$ denote the number of $k$-vertex subtrees of $F$.
\end{defi}

\begin{theorem}\label{main} For the maximum number of $k$-cycles $C_k$, we have  
$\exo(n, C_k)=
 g_{k-2}(T_{n-2})$  where $T_{n-2}$ is a BFS $3$-tree on $n-2$ vertices.
For every $k$ there are infinitely many values of $n$ for which
 the extremal graph is not unique, i.e. there are several different extremal graphs. 
\end{theorem}

\begin{proof}
To obtain the maximum number of cycles $C_k$ in an $n$ vertex outerplanar graph, we may clearly  suppose that the outerplanar graph is maximal. Each cycle subgraph $C_k$ is  triangulated to $k-3$ triangular faces, thus it corresponds to a subtree on $k-2$ vertices of the weak dual. Moreover, observe that this provides a one-to-one correspondence.  Hence our aim is to give a bound on the number of $k-2$-subtrees of the weak dual and describe the extremal graphs for which the number of $k-2$-subtrees is maximal in terms of $n$.

Székely and Wang initiated the investigation of counting the subtrees of trees in \cite{Szekely, Szekely2, Szekely3} and some general properties were described. They proved that those   binary trees which maximize the total number of subtrees are exactly the  the binary trees which
minimize the Wiener index, defined as the sum of all pairwise distances between vertices, and greedy trees have these properties. Note that this graph parameter is widely used in biochemistry.
In their paper \cite{Szekely}, the authors calculated the total number of subtrees of these extremal binary trees. We mention  that it had also a number theoretic aspect as well, as these formulas used a new representation of integers as a sum of powers of $2$.  The results were extended
to trees having a maximum degree constraint \cite{Kirk} and the extremal structures once again coincide with the ones
found by Székely and Wang. Later Yan and Yeh \cite{Yan} presented an algorithm for counting the number
of subtrees of a tree.

Following these footsteps, finally Andriantiana,  Wagner and  Wang proved that among trees having degrees less than of equal to $d$,  greedy trees maximize not only the total number of subtrees but the number of $k$-subtrees as well for every $k$    \cite{k-tree}.
This in turn implies that to prove a result on $\exo(n, C_k)$ it is sufficient to calculate the number of $k-2$-subtrees of the  greedy  BFS $3$-tree. \\ Concerning the the unicity of extremal graphs  suppose that $n$ is given in such a way that the greedy tree has $2^{\ell}+1$ leaves on the last level for some $\ell$ for which $\log_2{n}-2>\ell\geq \lceil k/2 \rceil$ holds. Consider the  leaf $w$ last reached during the BFS search. There is no $k$-subtree containing both $w$ and a further   leaf from the last level   due to the structure of the tree, as they have common ancestor only more than $\ell$ levels above. This implies that if we erase $w$, then no matter where we insert a further leave $w'$ instead of $w$, the number of trees will increase by the same number. However, the last common ancestor of the new leaf $w'$ and the other leaves on the last level may appear on any level apart from the lowest $\ell +1$ levels.
\end{proof}

As we mentioned in the Introduction, by Theorem \ref{main}  the extremal structure concerning  $\expo(n, C_k)$ can be derived from  the problem of maximizing the number subtrees in binary trees \cite{k-tree}. However the extremal value was not determined before. Thus we continue by presenting exact results  for short cycles and then we determine the corresponding asymptotics for each cycle length $k$ by proving Theorem \ref{asym} and providing a bound on the constant $c=c(k)$.

\begin{nota} Let $\mathcal{T}_n^{\leq d}$ denote  the family of trees on $n$ vertices having degrees at most $d$. The case $d=3$ is usually known as the family of $n$-vertex  binary trees. 
\end{nota}

\begin{prop} 
\begin{equation}\label{eq:1}
   \exo(n+2, C_3) =\max\{ g_1(G) \ | \ G \in \mathcal{T}_n^{\leq 3}\}=n \ \mbox { \ and \ }
   \exo(n+2, C_4)= \max\{ g_2(G)\  |\  G \in \mathcal{T}_n^{\leq 3}\}=n-1,
\end{equation}
and equality holds for every tree $\in \mathcal{T}_n^{\leq 3}$
\begin{equation}\label{eq:2}
   \exo(n+2, C_5)= \max\{ g_3(G) \  |\  G \in \mathcal{T}_n^{\leq 3}\}=\left\lfloor \frac{3n-6}{2}\right\rfloor,
\end{equation}
and equality holds for every tree having $\left\lfloor \frac{n-2}{2}\right\rfloor$ vertices of degree $3$,

\begin{equation}\label{eq:3}
   \exo(n+2, C_6)= \max\{ g_4(G) \  |\  G \in \mathcal{T}_n^{\leq 3}\}=\left\lfloor \frac{5n-18}{2}\right\rfloor,
\end{equation}

and equality holds for every tree having no vertices of degree $2$ or for trees with a single vertex of degree $2$, adjacent to a leaf.
\end{prop}

\begin{proof}
Equality in the cases of (\ref{eq:1}) in turn follows from the fact that the maximal value of $\exo(n+2, C_k)$ ($k\in \{3,4\}$) is attained on maximal outerplanar graphs, and such outerplanar graphs have exactly $2n+1$ edges from {which}  have $n$ triangular faces. Hence  according to Theorem \ref{main} it is enough to calculate the number of nodes, resp. edges in the weak dual, a tree on $n$ vertices.\\
Similarly to the previous case, (\ref{eq:2}) follows from counting $P_3$ subgraphs in binary trees. Suppose that a binary tree has $t_i$ vertices of degree $i$. Then the number of $P_3$ subgraphs is clearly $3t_3+t_2$, while $t_1+t_2+t_3=n$. Consequently, the number clearly increases if one can cut a branch from a degree $2$ vertex and glue it to another vertex of degree $2$. This means that we must have $t_3=\lfloor n/2\rfloor-1$, $t_1=\lfloor n/2\rfloor+1$ and the claim follows.\\ 
Case (\ref{eq:3}) can be proved  similarly by taking a tree transformation, i.e shifting branches among the tree like in the previous case,  and showing first that there cannot be more than $1$ vertices of degree $2$, then showing that the number of subtrees can be increased also if the degree $2$ vertex is not adjacent to a leaf. We leave the rest to the interested reader.  
\end{proof}

Next we show that $\exo(n, C_k)=(c(k)+o(1))n$ for some constant $c(k)$ depending only on $k$.  We introduce a simple lemma  to get this.

\begin{lemma}\label{Cata} Take an infinite $3$-regular tree. The number of $k$-vertex subtrees containing a fixed vertex $v$ equals $\frac{3}{2k+1}\binom{2k+1}{k-1}=\cC(k+1)-\cC(k)$ where $\cC(k)$ denotes the $k$th Catalan number.
\end{lemma}
\begin{remark} This integer sequence appears in several other enumeration problems, see \cite{OEIS}.
\end{remark}

\begin{proof}[Proof of Lemma \ref{Cata}]
Let us count first those $k$-vertex trees through a fixed vertex  $v$ which does not contain a fixed neighbour $w$ of $v$. We claim that their number is $\cC(k)$ and prove by induction. If we define the number of trees on zero vertices by $1$, then the statement clearly holds for $k=0$ and \mbox{$k=1$}, and in general it follows from the recursive formula of the Catalan numbers $\cC_k=\sum_{j=0}^{k-1} \cC_j\cC_{k-1-j}$. Indeed,  distinguishing between the cases describing the distribution  $(j,k-1-j : j=0,\ldots, k-1)$ of the number of vertices sitting on the first, resp. second branch joint to $v$, we get the desired formula by induction. \\
Using the previous subcase we are ready to prove the lemma. Suppose that a $k$ vertex tree containing $v$  has exactly $j$ vertices on the branch starting with $w$. Then the number of subtrees after deleting this branch is $\cC(k-j)$ according to our previous observation. On the other hand,  the number of branches having exactly $j$ vertices on the branch starting with $w$ is also a Catalan number, namely $\cC(j)$ by the same argument. Thus in total we have $$\sum_{j=0}^{k-1} \cC_j\cC_{k-j}=\sum_{j=0}^{k} \cC_j\cC_{k-j}-\cC_{k}$$ subtrees, completing the proof.
\end{proof}

\begin{proof}[Proof of Theorem \ref{asym}]
It is enough to show the statement for complete binary trees, that is, in those $3$-regular trees where the leaves of the tree are on the same level and every other vertex is of degree $3$. So suppose that $T$ is a $3$-regular greedy tree of $h+1$ levels  on $1+\sum_{i=0}^h 3\cdot 2^i$ vertices.

We introduce the notion  $t(k,r)$ for  the number of those $k$-vertex trees containing a fixed vertex $v$ from the lowest level of $T$ (having index $h$)  which have exactly $r$ vertices on the lowest level.

$ t(k,0)=0$ $\forall k$ as $v$ is on level $h$.

It is also clear that $t(0,r)=0$ $ \forall r$, $t(1,1)=1$, $t(1,r)=0$ if $r\ne 1$, and $t(k,r)=0$ if $r>k$.

Then we can determine $t(k,r)$ for all other values using the following recursion. 

 \begin{equation}\label{recursion1}
     t(k,r)=\sum_{s=1}^k t(k-r,s)\binom{2s-1}{r-1}.
 \end{equation}

To obtain (\ref{recursion1}), take the parent vertex $u$ of $v$ in the tree $T$ and erase each vertex of the $k$-vertex tree which is on level $h$. This way we get  $k-r$-vertex tree containing $u$.\\
Every $k-r$-vertex tree containing $u$ which has $s$ vertices on its lowest level $h-1$ can be expanded by $r$ vertices on  level $h$ exactly $\binom{2s-1}{r-1}$ different ways if we require that one of these $r$ vertices is $v$.
This enables us to determine the number of $k$-vertex trees having exactly $r$ vertices on level $h$ as $$3\cdot 2^h \cdot\frac{t(k,r)}{r}$$ since $v$ can take $3\cdot 2^h$ positions on level $h$ but each of these subtrees is counted $r$ times.

Let us define 
 $$c(k):=\sum_{r=1}^{k} \frac{t(k,r)}{r}.$$

By the above reasoning, exactly $3\cdot 2^g\cdot c(k)$ trees are reaching level $g$ but not having vertices on levels larger than $g$;  if $g$ is large enough (where $g>k$ clearly suffice.) Since the number of vertices on levels of index less than $k$ is constant, $c(k)$ is indeed the value we needed to determine.
\end{proof}

The recursion does not directly implies the order of magnitude of $c(k)$ in terms of $k$. We determine the growth of this function.

\begin{proof}[Proof of Theorem \ref{limit}]
First observe that every vertex of an infinite $3$-regular tree appears in at most $4^k$ subtrees of order $k$ in view of Lemma  \ref{Cata} and every subtree is counted $k$ times, thus $\sqrt[k]{c_k}< 4$ for all $k$.
On the other hand we will present a family of $k$-vertex subtrees such that the $k$-th root of the  cardinality  is tending to $4$. \\
  Let us fix a positive integer $q$ (to be determined later on) and take a root vertex $v$ on level $h_v:=h-q-\lceil\frac{k}{2^q}\rceil$ in a $3$-regular greedy tree of $h$ levels. Let us build a tree from $v$ by taking all the $2^{q+1}-2$ descendants on the next $q$ levels. Finally, let us root an arbitrary tree of size $\lfloor\frac{k}{2^q}\rfloor-2$ independently from each descendant of $v$ on level $h_v+q$. Note that every rooted tree is a subtree of the complete binary tree $T$ of $h$ levels since $\lfloor\frac{k}{2^q}\rfloor-2$ is less than the distance between the root and the lowest level $h$.
  Denoting $\lfloor\frac{k}{2^q}\rfloor-2$ by $\beta$, we have $\cC_{\beta}^{2^{q}}$ trees from $v$ and in total we have
  $$\cC_{\beta}^{2^{q}}\cdot 3\cdot 2^{h_{v}}$$ 
  such  trees with highest vertex being on level $h_{v}$ and they are of order $2^{q+1}-1+2^q\cdot (\lfloor\frac{k}{2^q}\rfloor-2)<k$. Obviously these trees can be completed to different $k$-subtrees of the underlying greedy tree by choosing added vertices outside the set  of descendants of the root vertex $v$.
  
  Suppose that $q:=\lceil \frac{1}{2}\log_2{k}\rceil$, i.e. $2^q \approx \sqrt k$. Then it is not hard to determine the order of magnitude of the constant $c(k)$  by  estimating the ratio of the number of the above trees and the number of vertices in the greedy tree.

  Here we use the standard estimates   $\cC_n \sim \frac{4^n}{n^{3/2}\sqrt{\pi}}$ of the Catalan numbers  to compute the ratio as

   $$\frac{\mbox{number \ of \ subtrees \ in \ consideration }}{n}\geq\frac{\cC_{\beta}^{2^{q}}\cdot 3\cdot 2^{h_{v}}}{3\cdot 2^{h+1}}>  \frac{4^k}{(2\cdot \sqrt{\pi})^{\sqrt{k}}\cdot k^{0.75(\sqrt{k}+1)}\cdot 4^{3\sqrt{k}}},$$
  and from this our theorem follows.
\end{proof}

\section{Maximal number of trees}
 In this section we present a proof for  Theorem \ref{paths} on the maximum number of paths $P_{k+1}$ with $k$ edges ($k$-paths) in outerplanar graphs and then derive Theorem \ref{outertree}. We use the standard notation $\cN(G, H)$ for the number of subgraphs isomorphic to $H$ in $G$, 
and
$G_1 \vee G_2$ for the join operation for  graphs $G_1$ and $ G_2$, where
$G_1 \vee G_2$ is obtained from the disjoint union of $G_1$ and $ G_2$ by connecting every pair of vertices $\{u,v\} \in V(G_1)\times V(G_2)$.

\subsection{Number of paths, upper bound}

\begin{nota}
Let $f(k)$  denote the maximum number of $k$-paths  with fixed endnodes in outerplanar graphs. We define $f(0):=1$.\\ Analogously,  let  $g(k)$  denote the maximum number of $k$-paths  with fixed endnodes in outerplanar graphs, where the endnodes are consecutive on the outer face.
\end{nota}

\begin{prop}
$$ f(k)\leq \sum_{r=1}^{k} 2\cdot g(r)f(k-r)  $$
\end{prop}

\begin{proof} We may suppose that the outerplanar graph is maximal. 
Let $u$ and $v$ be a fixed pair of vertices in an outerplanar graph $G$. We count the number of paths whose first vertex is $u$ and the $k+1$th vertex is $v$.

 \ \ \ {\em Case 1.} $u$ and $v$ are adjacent. \\ The edge $uv$ divides the graph $G$ into two outerplanar graphs $G_1$ and $G_2$ with a single common edge. Every path from $u$ to $v$ must be contained in one of $G_1$ and $G_2$ due to the outerplanar property of $G$. Hence $f(k)\leq 2g(k)$ in this case.
 
\ \ \  {\em Case 2.} $u$ and $v$ are not adjacent.\\
Let $x$ and $y$ the neighbours of $u$ which are the closest to $v$ on the outer face clockwise and counter-clockwise, respectively. The edge $ux$  divides $G$ to two outerplanar graphs. Let $G_1$ be the outerplanar graph obtained this way which does not contain $v$. Similarly, the  edge $uy$  divides $G$ to two outerplanar graphs. Let $G_2$ be the outerplanar graph obtained this way which does not contain $v$.
Then each $k$-path takes a few $1\leq r\leq k-1$ steps  either in $G_1$, eventually reaching $x$ and continuing to $v$, or  in $G_2$, eventually reaching $y$ and continuing to $v$. This yields the stated summation as an upper bound.
\end{proof}

\begin{lemma}
$g(r)\leq \cC({r-1})$. 
\end{lemma}

\begin{proof}
For $r=1$ and $r=2$ the statement clearly holds with equality. For $r>2$ consider two consecutive vertices $u$ and $v$ on the outer face in a maximal outerplanar graph $G$. Since $G$ is triangulated, there is a unique  vertex $w$ in the common neighbourhood $N(u)\cap N(v)$.  Let $G_1$ be the outerplanar graph which is not containing $v$ that we obtain from $G$ by cutting it into two halves via the edge $uw$. Similarly, let $G_2$ be the outerplanar graph which is not containing $u$ that we obtain from $G$ by cutting it into two halves via the edge $vw$. Each path from $u$ to $v$  of length $r>1$ passes through $w$ and consists of an $s$-path in $G_1$  and an $(r-s)$-path in $G_2$ for some $1\leq s\leq r-1$. Since the edges $uw$ and $wv$ are formed by consecutive vertices on the outer face of $G_1$ and $G_2$, we obtain an upper bound $ g(r)\leq \sum_{s=1}^{r-1}  g(s)g(r-s) $. This recursion coincides with the recursion of the Catalan numbers, hence the proof.
\end{proof}

As a corollary, we have
$$ f(k)\leq \sum_{r=1}^k 2\cdot \cC({r-1})f(k-r).$$ 

--------

\begin{proof}[Proof of Theorem \ref{paths}, upper bound]

To  prove the  statement   we show by induction $k$  that $f(k)\leq  4^k$    holds. First observe that this clearly holds for $k=0, 1$ and $2$ as outerplanar graphs are $K_{2,3}$-free graphs. Then using the inductional hypothesis we get

$$ f(k)\leq  \sum_{r=1}^k 2\cdot \cC({r-1})4^{k-r}.$$

In other words, if we take the generating function of the Catalan numbers $F(z)=\sum_{t=0}^{\infty} \cC({t})z^t$, we obtain an upper bound as
$$ f(k)< 4^{k-1} \cdot 2\cdot F(\frac{1}{4}).$$
On the other hand, we know that $F(z)=\frac{ 1-\sqrt{1-4z}}{2z}$, hence we get $f(k)<  4^{k-1}\cdot2\cdot2$, which is the desired bound.
\end{proof}

\begin{proof}[Proof of Theorem  \ref{outertree}, alternate proof on the number of trees]
We start by showing that  $\exo(n, T)= O(n^{\ell(T)})$ holds. Let us take an arbitrary set of $\ell(T)$ vertices from the vertex set $V(G)$ of any outerplanar graph $G$ on $n$ vertices. There are at most $\ell(T)!$ ways to identify these with the labelled leaves of $T$. Pick one of these identifications. We prove that the number of ways the identification can be completed to an embedding of $T$ in $G$ is bounded from above by a  constant depending on $\ell(T)$. Take one of the identified vertices $w$. In the proof of Theorem \ref{paths} we showed that there are constant number of ways to embed a path of given length between two fixed vertices, thus there is also a constant number of ways to provide such a combined embedding between $w$ and the other identified vertices. Every embedding of $T$ with leaves fixed in  $V(G)$ will be among these, hence the claim.

  $\exo(n, T)= \Omega(n^{\ell(T)})$ follows as a corollary from the following observation.
  \begin{prop}\label{also} $\exo(n, H)= \Omega(n^{\ell(H)})$, provided that $H$ is outerplanar.\end{prop}
\begin{proof} Take a positive constant $c<\frac{1}{\ell(H)}$ and replace each pendant edge of $H$ by a star on $\lfloor{cn}\rfloor$ vertices. The obtained outerplanar  can be completed to an $n$-vertex outerplanar graph  and it contains  at least $\lfloor{cn}\rfloor^{\ell(H)}$ copies of $H$ as a subgraph. \end{proof}\end{proof}

\subsection{Number of paths, lower bound: numeral system graphs}

The lower bound relies on an outerplanar graph which we call {\em numeral system graph}.

\begin{defi}[Numeral system graph]

For integers $N>1$ and $t>0$, we construct the graph $G(N,t)$ as follows. 

Let the vertices of $G(N,t)$ be  integers from $0$ to $N^t-1$ written in base $N$.
Each number $R$  is connected to  $(r-1)N^s$ and $(r+1)N^s$ if $R$ can be expressed as 
$R=rN^s$  for some  integers $0\le s\le t-1$ and $r\in \mathbb{Z}$. In this context, $0$ should be treated as the equivalent of $N^t$.

Furthermore, connect  every number  to the one  we obtain by changing its last nonzero digit to zero. See Figure \ref{Fig:Numsys-graph}.
\end{defi}

\begin{figure}[h]
\centering
\includegraphics[width=8cm]{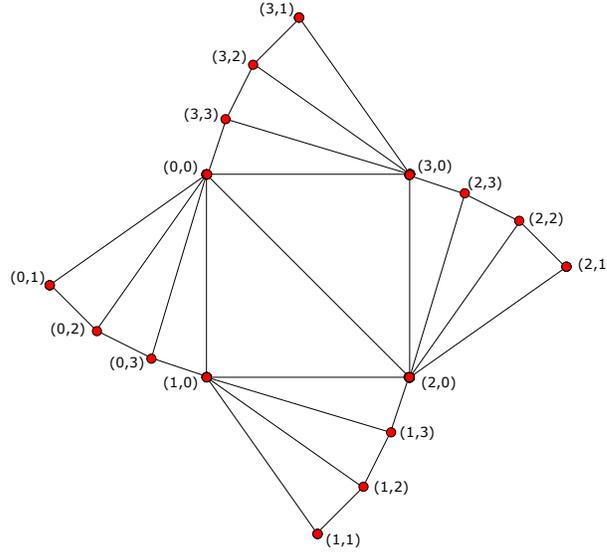}
\caption{The graph $G(4,2)$}
\label{Fig:Numsys-graph}
\end{figure}

It is easy to see by induction on $t$ that $G(N,t)$ is a triangulated outerplanar graph, moreover $G(N,1)$ is the same as $K_1 \vee P_{N-1}$.

We will prove that 
\begin{equation}\label{eq:5}
\lim_{k\rightarrow \infty} \lim_{N\rightarrow \infty}\sqrt[k]{ \frac{\cN(G(N,t), P_{k+1})}{\binom{|V(G(N,t))|}{2}}}\geq 4 \end{equation} for $t=\lfloor\sqrt{k}\rfloor$.

Indeed, this implies the lower bound in Theorem \ref{paths} since for each $n$ we can determine the largest $t$-th power $N^t$ which does not exceed $n$, and take a maximal outerplanar graph which  contains $G(N,t)$ as a subgraph. The distribution of the $t$-th powers imply that the lower bound $4$ on the limit still stands for arbitrary $n$.


The proof will be derived from a bijection between a subfamily of $k$-paths in $G(N,t)$ and a set of so-called  {\em permitted} sequences.

\begin{defi}
A sequence $[\gamma_1, \ldots, \gamma_{k-2t}] \in \mathbb{N}^{k-2t}$ is permitted (with parameters $k-2t$ and $t$) if (i) $\gamma_1=0$, (ii) $\gamma_{i+1}\leq \gamma_{i}+1$ $\forall i$, and (iii) $\gamma_{i}\leq t-2$ $\forall i$ hold.
\end{defi}

\begin{Lemma}\label{sequences1}
Suppose that  $m_0+m_1...+m_{t-2}=k-2t$ holds for a set of non-negative integers $m_i$ and assume that we are given a sequence 
 for all $0\le i\le t-3$  consisting of $m_i$ $i$-s  and $m_{i+1}$ $i+1$-s which starts with $i$.  
Then there exists a unique permitted sequence of parameter $k-2t$ and $t$ which contains all the given sequences as a subsequence.
\end{Lemma}

\begin{proof}

Note first that a subsequence of the numbers smaller than $j>0$ in a permitted sequence also forms a permitted sequence (with parameters $\sum_{i=0}^{j-1} m_i$, $j-1$) since the numbers never increase by more than $1$ in one step.
We prove the Lemma by induction on $t$. It clearly holds for $t=2$ for all $k\geq 2t$. Next, by the inductional hypothesis, we know that 
there is a unique permitted sequence of the numbers smaller than $j>0$ which contains all the given sequences as a subsequence for $0\le i\le j-2$, and  the subsequence of the numbers equal to $j-1$ or $j$ is given, which starts with $j-1$.
The position of $j$-s are thus determined  via the given subsequence of $m_{j-1} $ $j-1$-s and $m_j$ $j$-s in order to obtain a permitted sequence
as each $j$ must be inserted immediately after a $j-1$  or a $j$ following the order of the subsequence. \\
The sequence obtained this way will remain permitted.
\end{proof}

\begin{Cor}\label{Cor1}
The number of permitted sequences of parameters $k-2t$ and $t$ in which the number $i$ appears $m_i>0$ times for all $0\le i\le t-2$ is $\prod_{i=0}^{t-3} \binom{m_i+m_{i+1}-1}{m_i-1}$.
\end{Cor}

\begin{proof}
By  Lemma \ref{sequences1}, the number of such sequences equals  the product of the number of possible subsequences which consist of $i$-s of multiplicity $m_i$  and $i+1$-s of multiplicity $m_{i+1}$, 
and which start with $i$, for all $i$.
This yields $\binom{m_i+m_{i+1}-1}{m_i-1}$ for each $i<t-2$.
\end{proof}

\begin{prop}\label{sorozatszam}
Suppose that $t=\lfloor \sqrt{k}\rfloor\geq 4$ and denote by $\Gamma(k-2t, t)$ the number of permitted sequences of parameters $k-2t$ and $t$. Then we have $$\liminf_{k\to \infty} \sqrt[k]{\Gamma(k-2t, t)}\geq 4.$$
\end{prop}

\begin{proof}
Let $m_i=\lfloor \frac{k-2t}{t-2}\rfloor$ for all $0\leq i<t-2$ and $m_{t-2}= k-2t - (t-2)\lfloor \frac{k-2t}{t-2}\rfloor$. Then by Corollary \ref{Cor1} we have a lower bound on $\Gamma(k-2t, t)$ as $\Gamma(k-2t, t)>\prod_{i=0}^{t-3} \binom{m_i+m_{i+1}-1}{m_i-1}$. This in turn provides the bound 
$$\Gamma(k-2t, t)> \binom{2m_1-1}{m_1-1}^{t-2}>\frac{2^{(2m_1-1)(t-2)}}{(2m_1)^{t-2}}> \frac{2^{2k-5t}}{(2\sqrt{k})^{\sqrt{k}}},$$ from which the statement follows.
\end{proof}

\begin{proof}[Proof of the lower bound on $\exo(n, P_k)$ via inequality (\ref{eq:5})]
$ \ $ \newline
Let us choose an arbitrary pair of vertices $A=(\alpha_1, \ldots, \alpha_t)_N, B=(\beta_1, \ldots \beta_t)_N \in V(G(N,t))$ for which $\alpha_i, \beta_i\geq k \ \forall i$ and $\alpha_1\neq \beta_1$. We define an injective function between the set of permitted sequences of parameters $k-2t, t$ and the set of paths starting from $A$ and ending in $B$.

Consider a permitted sequence $[\gamma_1, \ldots, \gamma_{k-2t}]$, and let us denote the vertices of the path $A_0:=A, A_1, A_2, \ldots, A_k:=B$ defined iteratively according to $[\gamma_1, \ldots, \gamma_{k-2t}]$ as follows. \\
If $A_j$  is already defined for some $0\leq j<k-2t$ and $A_j$ ends in exactly $q$ zeros in base $N$ then 
\begin{itemize}
    \item obtain $A_{j+1}$ from $A_j$ by turning its last nonzero digit to zero if $\gamma_{j+1}=q+1$,
    \item $A_{j+1}:= A_j-N^s$ if $\gamma_{j+1}=s\leq q$.
\end{itemize}

Every digit of $A$ is at least $k$, and in every step we turn a digit to zero, or we decrease a digit by one and turn all the following digits  from $0$ to $N-1$. Since the second parameter of the sequence is $t$, the process will not affect the first digit in the first $k-2t$ steps.
Observe that $A_{j+1}$ will end in exactly $\gamma_{j+1}$ zeros. Indeed,    the last digit of $A_0$ is not zero, and  the property of having ending zeros as many as the corresponding element of the sequence is maintained during the process. 
 This implies that $A_0, \ldots A_{k-2t}$ indeed determines a path, and different sequences obviously determine different paths of length $k-2t$.

The final step to confirm the injection is to show that each such $(k-2t)$-path corresponding to a permitted sequence can be completed to a $k$-path ending in $B$.\\ From $A_{k-2t}$ we can reach the vertex $(0, 0, \ldots, 0, 0)_N$ in  $t'\leq t$ steps by turning each nonzero digit to zero from right to left. The next $t-1$ steps  correspond to fill in the digits of $B$ except for the last one, i.e. $(0, 0, \ldots, 0, 0)_N -(\beta_1, 0, \ldots, 0, 0)_N-(\beta_1, \beta_2, \ldots, 0, 0)_N -\ldots -(\beta_1, \beta_2, \ldots, \beta_{t-1}, 0)_N $. Finally we use $t'+1$ steps to reach $B$ by stepping to $(\beta_1, \beta_2, \ldots, \beta_{t-1}, \beta_t-t')_N $ and increasing the last digit in $t'$ steps by one.

From each ordered pair of vertices $A=(\alpha_1, \ldots, \alpha_t)_N, B=(\beta_1, \ldots \beta_t)_N \in V(G(N,t))$ for which $\alpha_i, \beta_i\geq k \ \forall i$ and $\alpha_1\neq \beta_1$, we obtained at least 
$\Gamma(k-2t, t)$ $k$-paths. Moreover the position of $(0,0, \ldots, 0)_N$ in the paths implies  these are different for each ordered pair. Since the number of such ordered pairs is $(N-k)^{2t-1}(N-k-1)$ and $\lim_{k\to \infty}\lim_{N\to \infty} \sqrt[k]{\frac{(N-k)^{2t-1}(N-k-1)}{\binom{N^t}{2}}}=1$, the proof follows from Proposition \ref{sorozatszam}. 
\end{proof}

\subsection{Number of paths: the case of small length}

In this subsection we present some partial results concerning the value of
$\exo(n, P_k)$ rather than its behaviour when both $n$ and $k$ are large, described in Theorem \ref{paths}.

\begin{theorem} Suppose that $n\geq 3.$ Then
$$ \exo(n, P_3)=\frac{n^2+3n-12}{2},$$
and the unique extremal graph  is $G(n, 1)= K_1 \vee P_{n-1}$.
\end{theorem}

\begin{proof}

First we calculate the number of $2$-paths in  $K_1 \vee P_{n-1}$.

For any graph $G$ this number is determined by the degree distribution  as $\sum_{v\in V(G)} \binom{\deg_G(v)}{2}$. Since the degree distribution is $n-1, 2, 2, 3, \ldots, 3$ in $K_1 \vee P_{n-1}$, we have $\cN( K_1 \vee P_{n-1},P_2)=\binom{n-1}{2}+(n-3)\binom{3}{2}+2\binom{2}{2}=\frac{n^2+3n-12}{2}$.

We prove the statement by induction on the number of vertices. Without loss of generality we may suppose that the extremal graph $G$ is a maximal outerplanar graph (MOP for brief), since every outerplanar graph can be completed to a MOP and there is at least one $2$-path on every  edge  in a MOP. For $n=3$, the triangle is the only only MOP, so here the statement holds. 

To perform the inductional step, assume that the claim holds for outerplanar graphs on $n-1$ vertices and consider a MOP  $G$ on $n$ vertices. The weak dual of $G$ is a tree, which has at least one leaf, so there is a triangular face $ABC$  in which both $AC$ and $BC$ lie on the outer face.

Remove the vertex $C$ and denote the remaining graph $G'$. It is still a MOP and the edge $AB$ is  on its outer face. 
A $2$-path in $G$ is either in $G'$ or it is composed of either $AC$ or $BC$ attached to an edge incident to this vertex in $G'$, or it is the path $ACB$ itself.

Thus $\cN(G,P_3)=\cN(G',P_3)+\deg_{G'}(A)+\deg_{G'}(B)+1$.

Since $G'$ is a triangulated outerplanar graph of $n-1$ vertices, and $AB$ is an edge on the outer face, the vertices $A$ and $B$ share exactly one common neighbour, that is,  
$\deg{G'}(A)+\deg{G'}(B)\le n$.

By the induction hypothesis, we have $\cN(G',P_3)\le \frac{(n-1)^2+3(n-1)-12}{2}=\frac{n^2+n-14}{2}$, hence we obtain $\cN(G,P_3)\le \frac{n^2+n-14}{2}+n+1=\frac{n^2+3n-12}{2}$. 

Equality  holds only if $G'$ is the unique extremal graph  $K_1 \vee P_{n-2}$ and every vertex is connected to either $A$ or $B$ in $G'= K_1 \vee P_{n-1}$. This is only possible if one of them has degree $n-2$ in $G'$, which means that $G=K_1 \vee P_{n-1}$, as claimed.
\end{proof}

For $P_4$, it is not true anymore that $K_1 \vee P_{n-1}$ is the extremal graph for each $n$.
Consider the case $n=6$. The maximal outerplanar graph in which the $2$nd, $4$th and $6$th vertex are connected to each outer has $33$ different $3$-paths, while $K_1 \vee P_5$ has only $32$.\\
This may seem a sporadic  counterexample   which coincides with the only counterexample for  the Cvetkovi\'c-Rowlinson conjecture \cite{BoNing}. However we show that this is not the case. 
In the following propositions we point out that $G(n,1)$ is not extremal apart from cases when $n$ and $k$ are small.

\begin{prop}\label{forgo_ut} Suppose that $n>k\geq 3$. Then
$$\cN(G(n,1), P_{k+1})=4(k-2)\binom{n+1-k}{2}+3(n-k)-1.$$ 
\end{prop}

\begin{proof}
Suppose first that the vertex $v_0$ with degree $n-1$ is not contained in the path. Then $P_{k+1}$ must be a subpath of $P_{n-1}$, hence we gain $n-k-1$  copies. Next suppose that $v_0$ is an endvertex of the path $P_{k+1}$. Then the indices of the rest of the vertices form an interval of $[1,n-1]$ and appear in increasing or decreasing order, thus we get $2(n-k)$ copies. Finally if $v_0$ is a vertex of $P_{k+1}$ but not an endvertex, then $P_{k+1}$ consists of an interval $I$ of $t\geq 1$ integers, $v_0$ and an interval $J$ of $k-t\geq 1$ integers, such that $I\cap J= \emptyset$ and the elements of $I$ precede the ones of $J$. In both interval, the elements can appear in increasing and decreasing order, independently. For each $1<t<k-1$, this provides $4\binom{n+1-k}{2}$ different path while for $t\in \{1, k-1\}$ we get only $2\binom{n+1-k}{2}$ as increasing and decreasing order does not differ for sequences of length one. This in turn completes the proof.
\end{proof}

\begin{prop} 
For $n/7>k>5$, $\exo(n, P_k)$ exceeds the number of $k$-paths in  $G(N,1)=K_1 \vee P_{n-1}$.
\end{prop}

\begin{proof}
Take the graph $G(n/3+1, 1)$ and consider two further copies of this graph $G'(n/3+1, 1)$ and $G''(n/3+1, 1)$. Identify the vertex pairs $\{v_0, v''_{n/3}\}$, $\{v'_0, v_{n/3}\}$, $\{v''_0, v'_{n/3}\}$ to obtain a triangulated outerplanar graph $G$ having $v_0v_0'v_0''$ as one of its triangular faces.
We estimate the value of $\cN(G, P_{k})$ as follows.\\
Suppose that $A$ and $B$ are vertices from different copies. We want to estimate the number of paths of length $k-1$ between $A$ and $B$. Without loss of generality may assume that $A\in G(n/3+1, 1)$ and $B\in G'(n/3+1, 1)$.\\
Consider those paths starting from $A$ and ending on $B$ which contains $v_0$ and $v'_0$ in the order $A-v_0-v'_0-B$ with  $A\in \{v_i : k-4<i<n/3-k+4\}$, $B\in \{v'_i : k-4<i<n/3-k+4\}$ ({\em admissible} endvertex pairs). Then there is exactly $2$ ways to insert either $1$ or more vertices between  $A-v_0$ or $v'_0-B$ while  there is exactly $2$ resp. $3$ ways to insert either $1$ resp. more vertices between  $v_0-v'_0$. This adds up to $2+3+2$ ways to obtain a path on $k$ vertices from $A$ to $B$ if the inserted vertices are between one consecutive pair of  $A-v_0-v'_0-B$. There are at least $\binom{3}{2}\cdot 2\cdot 2\cdot\binom{k-5}{1}$ ways to insert $k-4$ vertices  between  exactly two consecutive pair of the three consecutive pairs in  $A-v_0-v'_0-B$. Finally, there are at least  $\binom{3}{3}\cdot 2\cdot 2\cdot 2\cdot\binom{k-5}{2}$ ways to insert $k-4$ vertices  between  exactly three consecutive pair of the three consecutive pairs in  $A-v_0-v'_0-B$.\\
Hence in total, we get at least $4(k-4)^2+3$ paths of length $k>4$ between $A\in \{v_i : k-4<i<n/3-k+4\}$ and $B\in \{v'_i : k-4<i<n/3-k+4\}$. The number of admissible choices for the pair $A\in G(n/3+1, 1), B \in G'(n/3+1, 1)$ is at $(n/3-2k+8)^2$, hence taking into consideration the other two options we finally get at least $$(4(k-4)^2+3)\cdot 3\cdot (n/3-2k+8)^2$$ path subgraphs $P_k$ in $G$. If $n>7k$, it is easy to check that the latter expression exceeds to corresponding number of $P_{k}$ from Proposition \ref{forgo_ut}.
\end{proof}

It remains an open question the describe the extremal structures for paths.


\begin{thebibliography}{00}
\small
\parskip -1mm
 
 
 \bibitem{AlonS} Alon, N.,  Shikhelman, C. (2016). Many T copies in H-free graphs. Journal of Combinatorial Theory, Series B, 121, 146--172.
 
 
 \bibitem{AD-O}
Andriantiana, E. O. D., Dossou-Olory, A. A. V. (2020). Nordhaus-Gaddum inequalities for the number of connected induced subgraphs in graphs. arXiv preprint arXiv:2006.01187.
 
 \bibitem{k-tree}
Andriantiana, E. O. D., Wagner, S.,  Wang, H. (2013). Greedy trees, subtrees and antichains. the electronic journal of combinatorics, 20(3), P28.

\bibitem{extremaltrees}
Andriantiana, E. O., Misanantenaina, V. R.,  Wagner, S. (2020). Extremal trees with fixed degree sequence. arXiv preprint arXiv:2008.00722.

\bibitem{Char} Chartrand, G.,  Harary, F. (1967). Planar permutation graphs. In Annales de l'IHP Probabilités et statistiques (Vol. 3, No. 4, pp. 433--438).

\bibitem{Martin}
 Cox, C.,   R. Martin,  R. R. (2021) Counting paths, cycles and blow-ups in planar graphs,  Arxiv preprint 2101.05911.pdf

\bibitem{Olory} Dossou-Olory, A. A. (2018). Graphs and unicyclic graphs with extremal number of connected induced subgraphs. arXiv preprint arXiv:1812.02422.

 \bibitem{Eppstein}
Eppstein, D. (1993). Connectivity, graph minors, and subgraph multiplicity. Journal of Graph Theory, 17(3), 409--416.

 \bibitem{erdos}  Erd\H os, P. (1962). On the number of complete subgraphs contained in certain graphs. Magyar Tud. Akad. Mat. Kut. Int. K\"ozl.  459--474.
 
 \bibitem{P3}
 Győri, E., Paulos, A., Salia, N., Tompkins, C.,  Zamora, O. (2019). The Maximum Number of Paths of Length Three in a Planar Graph. arXiv preprint arXiv:1909.13539.
 
 

 
 \bibitem{Gyori}
Győri, E., Paulos, A., Salia, N., Tompkins, C.,  Zamora, O. (2020). Generalized Planar Tur\'an Numbers. arXiv preprint arXiv:2002.04579. 
 
 \bibitem{c4}
 Gy\H{o}ri, E.,   Paulos, A.,
 Zamora, O. The Minimum Number of $4$-Cycles in a Maximal Planar Graph with Small  Number of Vertices (2020), arXiv reprint arXiv:2005.12100. 
 

 \bibitem{c5} Győri, E., Paulos, A., Salia, N., Tompkins, C.,  Zamora, O. (2019). The maximum number of pentagons in a planar graph. arXiv preprint arXiv:1909.13532. 

 \bibitem{Hakimi} Hakimi,  S., Schmeichel E.F. On the Number of Cycles of Length $k$ in a Maximal Planar Graph. J. Graph Theory 3 (1979), 69--86. 
 
 \bibitem{Colin}
Van der Holst, H., Lovász, L.,  Schrijver, A. (1999). The Colin de Verdiere graph parameter. Graph Theory and Computational Biology (Balatonlelle, 1996), 29--85.

\bibitem{Huynh}  Huynh, T., Joret, G., Wood, D. R. (2020). Subgraph densities in a surface. arXiv preprint arXiv:2003.13777.

\bibitem{Huynh2} Huynh, T.,  Wood, D. R. (2020). Tree densities in sparse graph classes. arXiv preprint arXiv:2009.12989.

\bibitem{BoNing} Lin, H.,   Ning, B.
 A Complete Solution to the Cvetkovic-Rowlinson Conjecture (2021)
 \url{https://www.researchgate.net/publication/338253624_A_Complete_Solution_to_the_Cvetkovic-Rowlinson_Conjecture}

\bibitem{Kirk} Kirk, R.,  Wang, H. (2008). Largest number of subtrees of trees with a given maximum degree. SIAM Journal on Discrete Mathematics, 22(3), 985--995.

\bibitem{Nagy} Nagy, Z. L. (2018). Coupon-Coloring and total domination in Hamiltonian planar triangulations. Graphs and Combinatorics, 34(6), 1385--1394.

\bibitem{OEIS} On-line encyclopedia of integer sequences, http://oeis.org/A000245.

\bibitem{Szekely}
Székely, L.A., Wang, H.(2005). On subtrees of trees, Adv. Appl. Math. 34, 138--155.

\bibitem{Szekely2}
Székely, L. A.,  Wang, H. (2007). Binary trees with the largest number of subtrees. Discrete applied mathematics, 155(3), 374--385.

\bibitem{Szekely3} Székely, L. A.,  Wang, H. (2014). Extremal values of ratios: distance problems vs. subtree problems in trees II. Discrete Mathematics, 322, 36--47.

\bibitem{Tait} Tait, M. (2019). The Colin de Verdiere parameter, excluded minors, and the spectral radius. Journal of Combinatorial Theory, Series A, 166, 42--58.

 
\bibitem{Fisk} Urrutia, J. (2000). Art gallery and illumination problems. In Handbook of computational geometry (pp. 973--1027). North-Holland.

\bibitem{Wormald}
Wormald, N. C. (1986). On the frequency of 3-connected subgraphs of planar graphs. Bulletin of the Australian Mathematical Society, 34(2), 309--317.


\bibitem{Yan}  Yan W.G.,  Yeh, Y.-N. (2006). Enumeration of subtrees of trees, Theoret. Comput. Sci. 369, pp. 256--268.

\end{thebibliography}
\end{document}